\documentclass[preprint,12pt]{elsarticle}

\usepackage{amssymb}
\usepackage{amsmath}

\numberwithin{equation}{section}
\newtheorem{theorem}{Theorem}[section]
\newtheorem{corollary}[theorem]{Corollary}

\newtheorem{lemma}[theorem]{Lemma}

\newtheorem{question}[theorem]{Question}

\newproof{proof}{Proof}

\begin{document}

\begin{frontmatter}

\title{The inverse along a product and its applications}

\author[1,2]{Huihui Zhu}
\ead{ahzhh08@sina.com}
\author[2]{Pedro Patr\'{i}cio}
\ead{pedro@math.uminho.pt}
\author[1]{Jianlong Chen\corref{cor}}
\ead{jlchen@seu.edu.cn}
\cortext[cor]{Corresponding author}

\address[1]{Department of Mathematics, Southeast University, Nanjing 210096, China.}
\address[2]{CMAT-Centro de Matem\'{a}tica, Universidade do Minho, Braga 4710-057, Portugal.}

\begin{abstract}
 In this paper, we study the recently defined notion of the inverse along an element. An existence criterion for the inverse along
 a product is given in a ring. As applications, we present the equivalent conditions for the existence and expressions of the inverse along a matrix.
\end{abstract}


\end{frontmatter}


\section { \bf Introduction}

In this paper, $R$ is an associative ring with unity 1. An element $a\in R$ is (von Neumann) regular if there exists $x\in R$ such that $axa=a$. Such $x$ is called an inner inverse of $a$, denoted by $a^{-}$. An arbitrary inner inverse of $a$ is denoted by $a^{(1)}$. We call $b$ an outer inverse of $a$ provided that $bab=b$. If $b$ is both an inner and an outer inverse of $a$, then it is a reflexive inverse of $a$, and is denoted by $a^{+}$.

Given a semigroup $S$, $S^1$ denotes the monoid generated by $S$. Following Green \cite{Green}, Green's preorders and relations in a semigroup are defined by

$a\leq_\mathcal{L}b \Leftrightarrow S^1a \subset S^1b \Leftrightarrow$ there exists $x\in S^1$ such that $a=xb$.

$a\leq_\mathcal{R}b \Leftrightarrow aS^1 \subset bS^1 \Leftrightarrow$  there exists $x\in S^1$ such that $a=bx$.

$a\leq_\mathcal{H}b \Leftrightarrow a\leq_\mathcal{L}b ~~{\rm and}~~ a\leq_\mathcal{R}b$.

$a\mathcal{L}b\Leftrightarrow S^1a = S^1b \Leftrightarrow$ there exist $x,y\in S^1$ such that $a=xb$ and $b=ya$.

$a\mathcal{R}b\Leftrightarrow aS^1 = bS^1 \Leftrightarrow$ there exist $x,y\in S^1$ such that $a=bx$ and $b=ay$.

$a\mathcal{H}b \Leftrightarrow a\mathcal{L}b ~~{\rm and}~~ a\mathcal{R}b$.

Recently, Mary \cite{Mary} introduced the notion of the inverse along an element that is based on Green's relation in a semigroup $S$. Given $a,d\in S$, an element $a\in S$ is invertible along $d$ \cite{Mary} if there exists $b$ such that $dab=d=bad$ and $b\leq_\mathcal{H}d$. The element $b$ above is unique if it exists, and is denoted by $a^{\parallel d}$. Recall that $a^{\parallel d}$ exists implies that $d$ is regular. Later, Mary and Patr\'{i}cio \cite{Mary and Patricio} proved that $a$ is invertible along $d$ if and only if $d\mathcal{H}dad$, which gave a new existence criterion for the inverse along an element. Further, given a regular element $d$, they \cite{Mary and Patricio, Mary and Patricio 2} characterized the existence of $a^{\parallel d}$ by means of a unit and $d^{-}$ in a ring.  Moreover, the representation of $a^{\parallel d}$ is given. As applications, they \cite{Mary and Patricio 2} derived the equivalent conditions for the existence and the formula of the inverse along a regular lower triangular matrix. More results on the inverse along an element can be found in mathematical literature \cite{Mary E, Serbia}.

Motivated by papers \cite{Mary and Patricio, Mary and Patricio 2}, we investigate the inverse along a product $pmq$ ($m$ is regular) in a ring, extending the results in \cite{Mary and Patricio, Mary and Patricio 2}. As applications, the inverse along a regular matrix
$\begin{bmatrix}
              d_1 & d_3 \\
              d_2 & d_4
        \end{bmatrix}$ \normalsize is given under some conditions.

\section{The inverse along a product $pmq$}

In this section, we begin with some lemmas which play important roles in the sequel.

\begin{lemma} \label{Jlemma} Given $a, b\in R$, then $1+ab$ is invertible if and only if $1+ba$ is invertible. Moreover, $(1+ba)^{-1} = 1-b(1+ab)^{-1}a$.
\end{lemma}

Lemma \ref{Jlemma} is known as the Jacobson's Lemma (see e.g. \cite{Jacobson}).

\begin{lemma}\label{eRe} {\rm (\cite[Theorem 1]{Patricio and Puystjens 2})} Let $R$ be a ring and $e$ an idempotent in $R$. Then $exe+1-e$ is invertible in $R$ if and only if $exe$ is invertible in $eRe$.
\end{lemma}

The next theorem, a main result of this paper, gives an existence criterion of the inverse along a product $pmq$ in a ring.

\begin{theorem} \label{mary inverse} Let $p,a,q,m\in R$ with $m$ regular. If $m\leq_\mathcal{L}pm$ and $m\leq_\mathcal{R}mq$, then the following conditions are equivalent

\emph{(i)} $a$ is invertible along $pmq$.

\emph{(ii)} $u=mqap +1- mm^{(1)}$ is invertible.

\emph{(iii)} $v=qapm+1-m^{(1)}m$ is invertible.\\

In this case, $$a^{\parallel pmq}=pu^{-1}mq=pmv^{-1}q.$$
\end{theorem}

\begin{proof} It follows from Lemma \ref{Jlemma} that $(ii)\Leftrightarrow(iii)$. Next, it is sufficient to prove $(i)\Leftrightarrow(ii)$.

$(i)\Rightarrow(ii)$ Suppose that $a$ is invertible along $pmq$. From $m\leq_\mathcal{L}pm$ and $m\leq_\mathcal{R}mq$, then there exist $p'$ and $q'$ such that $p'pm=m=mqq'$. In view of \cite[Theorem 2.2]{Mary and Patricio}, we know that $a$ is invertible along $pmq$ if and only if $pmq \mathcal{H} pmqapmq$. There are $x,y\in R$ such that $$pmq=xpmqapmq=pmqapmqy.\eqno(1)$$
Multiplying the above equation (1) by $p'$ on the left yields
 \begin{center}
 $mq=mqapmqy$.
\end{center}
Multiplying the above equation (1) by $q'$ on the right yields
\begin{center}
$pm=xpmqapm$.
\end{center}
 Hence, $$mqapmm^{(1)}(mqyq'm^{(1)}mm^{(1)})=mm^{(1)}=(mm^{(1)}p'xpmm^{(1)})mqapmm^{(1)}.$$
The equalities above show that $mqapmm^{(1)}$ is invertible in $mm^{(1)}Rmm^{(1)}$.
By Lemma \ref{eRe}, $mqapmm^{(1)} +1- mm^{(1)}$ is invertible in $R$. Again, Lemma \ref{Jlemma} ensures that $u=mqap +1- mm^{(1)}$ is invertible.

$(ii)\Rightarrow(i)$ Suppose that $u$, therefore $v$ are invertible. From $um=mv=mqapm$, it follows that $pmq=pu^{-1}mqapmq=pmqapmv^{-1}q$ and $pu^{-1}mq=pmv^{-1}q$.
Pose $b=pu^{-1}mq=pmv^{-1}q$, then $b\leq_\mathcal{H}pmq$ since $pu^{-1}mq=pu^{-1}p'pmq=pmqq'v^{-1}q$.

Hence, $a$ is invertible along $pmq$. Moreover,
\begin{center}
$a^{\parallel pmq}=pu^{-1}mq=pmv^{-1}q$.
\end{center}

 The proof is completed.\hfill$\Box$
\end{proof}

If $p$ is left invertible and $q$ is right invertible, then $m\mathcal{L}pm$ and $m\mathcal{R}mq$. As a special result of Theorem \ref{mary inverse}, we have the following corollary.

\begin{corollary} Let $p,a,q,m\in R$ with $m$ regular. If $p$ is left invertible and $q$ is right invertible, then the following conditions are equivalent

\emph{(i)} $a$ is invertible along $pmq$.

\emph{(ii)} $u=mqap +1- mm^{(1)}$ is invertible.

\emph{(iii)} $v=qapm+1-m^{(1)}m$ is invertible.\\

In this case, $$a^{\parallel pmq}=pu^{-1}mq=pmv^{-1}q.$$
\end{corollary}

Taking $p=q=1$, we get

\begin{corollary} \label{m and p theorem}{\rm (\cite[Theorem 3.2]{Mary and Patricio} and \cite[Theorem 1.3]{Mary and Patricio 2})} Let $m$ be a regular element of a ring $R$. Then the following are equivalent

\emph{(i)} $a$ is invertible along $m$.

\emph{(ii)} $u=ma+1-mm^{(1)}$ is invertible.

\emph{(iii)} $v=am+1-m^{(1)}m$ is invertible.\\

In this case, $$a^{\parallel m}=u^{-1}m=mv^{-1}.$$
\end{corollary}

\section{Applications to the inverse along a matrix}

Mary, Patr\'{i}cio \cite{Mary and Patricio 2} gave some equivalent conditions for the existence of the inverse along a regular lower triangular matrix $\begin{bmatrix}
              d_1 & 0 \\
              d_2 & d_4
        \end{bmatrix}$ over a Dedekind-finite ring. \normalsize It would be interesting to find the related existence criteria and formula of the inverse along a regular matrix
$D=\begin{bmatrix}
              d_1 & d_3 \\
              d_2 & d_4
        \end{bmatrix}$, in the general case.

By $R_{2 \times 2}$ we denote the ring of $2 \times 2$ matrices over $R$. Let $D=
\begin{bmatrix}
              d_1 & d_3 \\
              d_2 & 0
        \end{bmatrix}\in R_{2 \times 2}$ and $D=
\begin{bmatrix}
              0 & 1 \\
              1 & 0
\end{bmatrix}
 \begin{bmatrix}
              d_2 & 0\\
              d_1 & d_3
        \end{bmatrix}
        \begin{bmatrix}
              1 & 0 \\
              0 & 1
        \end{bmatrix}=:PMQ.$ Given a lower triangular matrix $M=\begin{bmatrix}
              d_2 & 0\\
              d_1 & d_3
        \end{bmatrix}$ with $d_2$ and $d_3$ regular, Patr\'{i}cio and Puystjens \cite{Patricio and Puystjens} proved that $M$
 is regular if and only if $w=(1-d_3d_3^{+})d_1(1-d_2^{+}d_2)$ is regular. In this case, $$MM^{-}=\begin{bmatrix}
              d_2d_2^{+} ~&~ 0 \\
             (1-ww^{-})(1-d_3d_3^{+})d_1d_2^{+}  ~&~ d_3d_3^{+}+ww^{-}(1-d_3d_3^{+})
        \end{bmatrix}.$$

 First, we consider the inverse along a regular $(2,2,0)$ matrix in a ring.

 \begin{theorem}  Let $A=\begin{bmatrix}
              a & c \\
              b & d
        \end{bmatrix}$, $D=\begin{bmatrix}
              d_1 & d_3 \\
              d_2 & 0
        \end{bmatrix}\in R_{2 \times 2}$ with $d_3$ and $d_3$ regular.
 If $c^{\parallel d_2}$ exists,
 then $A^{\parallel D}$ exists if and only if $\xi=\beta-\alpha c^{\parallel{d_2}}a$ is invertible.

 In this case, $A^{\parallel D} =
\begin{bmatrix}
              \xi^{-1}(d_1-\alpha c^{\parallel{d_2}})&~~ \xi^{-1}d_3 \\
             c^{\parallel{d_2}}[1-a\xi^{-1}(d_1-\alpha c^{\parallel{d_2}})]  &~~ -c^{\parallel{d_2}}a\xi^{-1}d_3
        \end{bmatrix}$, where
 \begin{eqnarray*}
 \alpha & = & d_1c+d_3d -(1-ww^{-})(1-d_3d_3^{+})d_1d_2^{+},\\
 \beta & = & d_1a+d_3b+(1-ww^{-})(1-d_3d_3^{+}),\\
 \xi   & = & \beta-\alpha c^{\parallel{d_2}}a.
 \end{eqnarray*}
 \end{theorem}

\begin{proof} We have $MAP=
\begin{bmatrix}
              d_2c & d_2a \\
              d_1c+d_3d & d_1a+d_3b
        \end{bmatrix}$. Hence,
 \begin{center}
 $U=MAP+I-MM^{-}=\begin{bmatrix}
              u & d_2a \\
             \alpha  & \beta
        \end{bmatrix}$, where
 \end{center}

 \begin{eqnarray*}
 u & = & d_2c+1-d_2d_2^{+},\\
 \alpha & = & d_1c+d_3d -(1-ww^{-})(1-d_3d_3^{+})d_1d_2^{+},\\
 \beta & = & d_1a+d_3b+(1-ww^{-})(1-d_3d_3^{+}).
 \end{eqnarray*}

Since $c^{\parallel d_2}$ exists, it follows that $u=d_2c+1-d_2d_2^{+}$ is invertible and $c^{\parallel d_2}=u^{-1}d_2$. Using Schur complements we get the factorization
$$U=\begin{bmatrix}
              1 & 0 \\
             \alpha u^{-1} & 1
        \end{bmatrix}
 \begin{bmatrix}
              u & 0 \\
             0  & \xi
        \end{bmatrix}
 \begin{bmatrix}
              1 & c^{\parallel{d_2}} a \\
             0  & 1
        \end{bmatrix},$$ where $\xi=\beta-\alpha c^{\parallel{d_2}} a$.
 Hence, $U$ is invertible if and only if $\xi$ is invertible.

Note that $U^{-1}= \begin{bmatrix}
              1 & -c^{\parallel{d_2}}a \\
             0  & 1
        \end{bmatrix}
 \begin{bmatrix}
              u^{-1} & 0 \\
             0  & \xi^{-1}
        \end{bmatrix}
 \begin{bmatrix}
              1 & 0 \\
             -\alpha u^{-1} & 1
        \end{bmatrix}
.$
Then $$A^{\parallel D}=PU^{-1}M =
\begin{bmatrix}
              \xi^{-1}(d_1-\alpha c^{\parallel{d_2}})&~~ \xi^{-1}d_3 \\
             c^{\parallel{d_2}}[1-a\xi^{-1}(d_1-\alpha c^{\parallel{d_2}})]  &~~ -c^{\parallel{d_2}}a\xi^{-1}d_3
        \end{bmatrix}.$$

 The proof is completed. \hfill$\Box$
\end{proof}

Now, suppose that $d_4$ in the matrix $D$ is regular and set $e=1-d_4d_4^{+}$, $f=1-d_4^{+}d_4$ and $s=d_1-d_3d_4^{+}d_2$. We have the following decomposition
\begin{center}
$D=\begin{bmatrix}
              d_1 & d_3 \\
              d_2 & d_4
        \end{bmatrix} = \begin{bmatrix}
              1 & d_3d_4^{+}\\
              0 & 1
        \end{bmatrix}
\begin{bmatrix}
               s  & &d_3f\\
              ed_2 & &d_4
        \end{bmatrix}
\begin{bmatrix}
               1      & 0\\
              d_4^{+}d_2 & 1
        \end{bmatrix}=:PMQ.$
\end{center}

We next discuss the inverse of $A=\begin{bmatrix}
              a & c \\
              b & d
        \end{bmatrix}$ along a regular matrix $D$, under certain conditions.

\begin{theorem} \label{Along a Matrix}

Let $A=\begin{bmatrix}
              a & c \\
              b & d
        \end{bmatrix}$, $D=\begin{bmatrix}
              d_1 & d_3 \\
              d_2 & d_4
       \end{bmatrix}\in R_{2 \times 2}$ with $d_4$ regular.
 With the notations above, if $d_3f=0$ and $a^{\parallel s}$ exists, then $A^{\parallel D}$ exists
 if and only if $\xi=\beta-\alpha a^{\parallel s}(ad_3d_4^{+}+c)$ is invertible.

 In this case, $A^{\parallel D}=
\begin{bmatrix}
x_1d_2+ x_2s &~~ x_1d_4 \\
\xi^{-1}(d_2-\alpha a^{\parallel s}) &~~ \xi^{-1}d_4
        \end{bmatrix}$, where
 \begin{eqnarray*}
 u   & = & sa+1-ss^{+},\\
 t  & = & ed_2(1-s^{+}s),\\
 \alpha  & = &  d_2a+d_4b-(1-tt^{-})ed_2s^{+}, \\
 \beta & = & (d_2a+d_4b)d_3d_4^{+}+d_2c+d_4d+(1-tt^{-})e,\\
 \xi & = & \beta-\alpha a^{\parallel s}(ad_3d_4^{+}+c),\\
 x_1 & = &[(1 -a^{\parallel s}a)d_3d_4^{+}-a^{\parallel s}c]\xi^{-1},\\
 x_2  & = &  u^{-1}-x_1 \alpha u^{-1}.
\end{eqnarray*}
\end{theorem}

\begin{proof} If $d_3f=0$, then $M=
\begin{bmatrix}
               s  & &0\\
              ed_2 & &d_4
  \end{bmatrix}$. Note that the regularity of $D$ is equivalent to the regularity of $M$. Hence, it follows from \cite[Theorem 1]{Patricio and Puystjens} that $$I-MM^{-}=\begin{bmatrix}
               1-ss^{+}  & &~~~0\\
              -(1-tt^{-})ed_2s^{+} & &~~~(1-tt^{-})e
        \end{bmatrix},$$ where $t=ed_2(1-s^{+}s)$.

Note that $MQAP=
\begin{bmatrix}
              sa &~ s(ad_3d_4^{+}+c) \\
              d_2a+d_4b &~ (d_2a+d_4b)d_3d_4^{+}+d_2c+d_4d
        \end{bmatrix}.$ We have

 $$U=MQAP+I-MM^{-}=
\begin{bmatrix}
              u &~ s(ad_3d_4^{+}+c) \\
              \alpha &~ \beta
        \end{bmatrix},$$
 where
 \begin{eqnarray*}
 u   & = & sa+1-ss^{+},\\
 \alpha  & = &  d_2a+d_4b-(1-tt^{-})ed_2s^{+}, \\
\beta & = & (d_2a+d_4b)d_3d_4^{+}+d_2c+d_4d+(1-tt^{-})e.
 \end{eqnarray*}
In this case,
$$U^{-1}=\begin{bmatrix}
              1 &~ -a^{\parallel s}(ad_3d_4^{+}+c) \\
              0 &~1
        \end{bmatrix}
\begin{bmatrix}
              u^{-1} & 0 \\
              0 & \xi^{-1}
        \end{bmatrix}
        \begin{bmatrix}
              1 & 0 \\
              -\alpha u^{-1} & 1
        \end{bmatrix},$$ where $\xi=\beta-\alpha a^{\parallel s}(ad_3d_4^{+}+c)$.

 By calculations, $A^{\parallel D}=PU^{-1}MQ=\begin{bmatrix}
              x_1d_2+ x_2s &~~~ x_1d_4 \\
              \xi^{-1}(d_2-\alpha a^{\parallel s}) &~~~ \xi^{-1}d_4
        \end{bmatrix}$, where
 \begin{eqnarray*}
 x_1 & = &[(1 -a^{\parallel s}a)d_3d_4^{+}-a^{\parallel s}c]\xi^{-1},\\
 x_2  & = &  u^{-1}-x_1 \alpha u^{-1}.
\end{eqnarray*}

The proof is completed. \hfill$\Box$
\end{proof}

If $d_4$ is invertible, then $e=f=0$. Hence, we have the following corollary.

\begin{corollary} Let $A=\begin{bmatrix}
              a & c \\
              b & d
       \end{bmatrix}$, $D=\begin{bmatrix}
              d_1 & d_3 \\
              d_2 & d_4
       \end{bmatrix}\in R_{2 \times 2}$ with $d_4$ invertible.
 If $a^{\parallel s}$ exists, then $A^{\parallel D}$ exists
 if and only if $\xi=\beta-\alpha a^{\parallel s}(ad_3d_4^{-1}+c)$ is invertible.

 In this case, $A^{\parallel D}=\begin{bmatrix}
               x_1d_2+ x_2s &~~~ x_1d_4 \\
              \xi^{-1}(d_2-\alpha a^{\parallel s}) &~~~ \xi^{-1}d_4
        \end{bmatrix}$, where
 \begin{eqnarray*}
 s   & = & d_1-d_3d_4^{-1}d_2,\\
 u   & = & sa+1-ss^{+},\\
 \alpha  & = &  d_2a+d_4b, \\
\beta & = & \alpha d_3d_4^{-1}+d_2c+d_4d,\\
 \xi & = & \beta-\alpha a^{\parallel s}(ad_3d_4^{-1}+c),\\
 x_1 & = &[(1 -a^{\parallel s}a)d_3d_4^{-1}-a^{\parallel s}c]\xi^{-1},\\
 x_2  & = &  u^{-1}-x_1 \alpha u^{-1}.
\end{eqnarray*}
\end{corollary}

In Theorem \ref{Along a Matrix}, take $d_3=0$, then $s=d_1$.  We can get the formula and equivalence for the existence of the inverse along a regular lower triangular matrix obtained in \cite{Mary and Patricio 2}.

\begin{corollary} {\rm (\cite[Theorem 3.1]{Mary and Patricio 2})} Let $A=\begin{bmatrix}
              a & c \\
              b & d
        \end{bmatrix}$, $D=\begin{bmatrix}
              d_1 & 0 \\
              d_2 & d_4
        \end{bmatrix}\in R_{2 \times 2}$ with $d_4$ regular.
 With the notations above, if $a^{\parallel d_1}$ exists, then $A^{\parallel D}$ exists
 if and only if $\xi=\beta-\alpha a^{\parallel d_1}c$ is invertible.

 In this case, $A^{\parallel D}=\begin{bmatrix}
              a^{\parallel d_1} &~~~-a^{\parallel d_1}c\xi^{-1}d_4\\
              \xi^{-1}(d_2-\alpha a^{\parallel d_1}) &~~~ \xi^{-1}d_4
        \end{bmatrix}$, where
 \begin{eqnarray*}
 u   & = & d_1a+1-d_1d_1^{+},\\
 t  & = & ed_2(1-d_1^{+}d_1),\\
 \alpha  & = &  d_2a+d_4b-(1-tt^{-})ed_2d_1^{+}, \\
 \beta & = & d_2c+d_4d+(1-tt^{-})e,\\
 \xi  & = & \beta-\alpha a^{\parallel d_1}c.
\end{eqnarray*}
\end{corollary}

By taking $ed_2=0$ in Theorem \ref{Along a Matrix}, we get the following corollary.

\begin{corollary}
Let $A=\begin{bmatrix}
              a & c \\
              b & d
        \end{bmatrix}$, $D=\begin{bmatrix}
              d_1 & d_3 \\
              d_2 & d_4
        \end{bmatrix}\in R_{2 \times 2}$ with $d_4$ regular.
 With the notations above, if $ed_2=d_3f=0$ and $a^{\parallel s}$ exists, then $A^{\parallel D}$ exists
 if and only if $\xi=\beta-\alpha a^{\parallel s}(ad_3d_4^{+}+c)$ is invertible.

 In this case, $A^{\parallel D}=\begin{bmatrix}
               x_1d_2+ x_2s &~~~ x_1d_4 \\
              \xi^{-1}(d_2-\alpha a^{\parallel s}) & ~~~\xi^{-1}d_4
        \end{bmatrix}$, where
 \begin{eqnarray*}
 u   & = & sa+1-ss^{+},\\
 \alpha  & = &  d_2a+d_4b, \\
 \beta & = & \alpha d_3d_4^{+}+d_2c+d_4d+e,\\
 \xi & = & \beta-\alpha a^{\parallel s}(ad_3d_4^{+}+c),\\
  x_1 & = &[(1 -a^{\parallel s}a)d_3d_4^{+}-a^{\parallel s}c]\xi^{-1},\\
 x_2  & = &  u^{-1}-x_1 \alpha u^{-1}.
\end{eqnarray*}
\end{corollary}

\begin{question} {\rm Given a regular matrix $D$, can we give further equivalent conditions such that $A^{\parallel D}$ exists without additional conditions ?}
\end{question}

\bigskip

\centerline {\bf ACKNOWLEDGMENTS}
\vskip 2mm

The first author is grateful to China Scholarship Council for supporting him to purse his further study in University of Minho, Portugal. This research is supported by the FEDER Funds through ¡®Programa Operacional Factores de Competitividade-COMPETE', the Portuguese Funds through FCT- `Funda\c{c}\~{a}o para a Ci\^{e}ncia e Tecnologia', within the project PEst-OE/MAT/UI0013/2014, the National Natural Science Foundation of China (No. 11201063 and No. 11371089), the Specialized Research Fund for the Doctoral Program of Higher Education (No. 20120092110020), the Natural Science Foundation of Jiangsu Province (No. BK20141327), the Foundation of Graduate Innovation Program of Jiangsu Province(No. CXLX13-072) and the Fundamental Research Funds for the Central Universities (No. 22420135011).

\begin{flushleft}
{\bf References}
\end{flushleft}

\end{document}